\documentclass[12pt]{amsart}
\usepackage{graphicx}
\usepackage{amssymb}
\usepackage{amsmath}
\usepackage{amsthm,amsfonts,bbm}
\usepackage{amscd}
\usepackage{geometry}
\usepackage{epsfig,epstopdf}
\usepackage{mathrsfs}
\usepackage[backref]{hyperref}
\usepackage{color}

\newtheorem{theorem}{Theorem}[section]
\newtheorem{corollary}[theorem]{Corollary}
\newtheorem{lemma}[theorem]{Lemma}
\theoremstyle{definition}
\newtheorem{definition}[theorem]{Definition}

\newtheorem{fact}{\bf Fact}
\newtheorem{problem}{\bf Problem}

\numberwithin{equation}{section}
\numberwithin{figure}{section}
\def\A{\mathcal{A}}
\def\B{\mathcal{B}}
\def\C{\mathbb{C}}
\def \ex{\mathrm{ex}}
\def \F{\mathcal{F}}
\def\la{\lambda}
\def\p{\partial}
\def \spex{\mathrm{spex}}
\def \x{\mathbf{x}}

\begin{document}
\title[Spectral Tur\'an Problems]{Spectral bipartite Tur\'{a}n problems on linear hypergraphs}

\author[C.-M. She]{Chuan-Ming She}
\address{Center for Pure Mathematics, School of Mathematical Sciences, Anhui University, Hefei 230601, P. R. China}
\email{shecm@stu.ahu.edu.cn}

\author[Y.-Z. Fan]{Yi-Zheng Fan*}
\address{Center for Pure Mathematics, School of Mathematical Sciences, Anhui University, Hefei 230601, P. R. China}
\email{fanyz@ahu.edu.cn}
\thanks{*The corresponding author. This work was supported by National Natural Science Foundation of China (No. 12331012).
}

\author[L. Kang]{Liying Kang}
\address{Department of Mathematics, Shanghai University, Shanghai 200444, P. R.  China}
\email{lykang@shu.edu.cn}

\subjclass[2000]{05C35, 05C65}

\keywords{Linear hypergraph; spectral extreme problem; adjacency tensor; spectral radius; Berge hypergraph}
	
\begin{abstract}
Let $F$ be a graph, and let $\B_r(F)$ be the class of $r$-uniform Berge-$F$ hypergraphs.
In this paper, we establish a relationship between the spectral radius of the adjacency tensor of a uniform hypergraph and its local structure through walks.
Based on the relationship, we give a spectral asymptotic bound for $\B_{r}(C_3)$-free linear $r$-uniform hypergraphs and upper bounds for the spectral radii of $\B_{r}(K_{2,t})$-free or $\{\B_{r}(K_{s,t}),\B_{r}(C_{3})\}$-free linear $r$-uniform hypergraphs, where $C_{3}$ and $K_{s,t}$ are respectively the triangle and the complete bipartite graph with one part having $s$ vertices and the other part having $t$ vertices.
 Our work implies an upper bound for the number of edges of $\{\B_{r}(K_{s,t}),\B_{r}(C_{3})\}$-free linear $r$-uniform hypergraphs and extends some of the existing research on (spectral) extremal problems of hypergraphs.
\end{abstract}

\maketitle
	
\section{Introduction}
A \emph{hypergraph} $H=(V(H),E(H))$ consists of a vertex set $V(H)$ and an edge set $E(H)$, where each edge of $E(H)$ is a subset of $V(H)$.
The hypergraph $H$ is called \emph{$r$-uniform} if each edge has exactly $r$ elements,
and is called \emph{linear} if any two edges intersect into at most one vertex.
Clearly, a simple graph is a $2$-uniform linear hypergraph.
A \emph{walk} $W$ of length $k$ in $H$, simply called a \emph{$k$-walk},
  is an alternating sequence of vertices and edges of the form $v_{1}e_{1}v_{2}e_{2}v_{3}\cdots v_{k}e_{k}v_{k+1}$,
where $v_{i}\ne v_{i+1}$ and $\left\{v_{i},v_{i+1}\right\}\subseteq e_{i}$ for $i \in [k]:=\{1,\ldots,k\}$.
The hypergraph $H$ is \emph{connected} if every two vertices are connected by a walk.
In the above walk $W$, if no vertices or edges are repeated in the sequence, then $W$ is called a \emph{Berge path};
if $v_1=v_{k+1}$ and except $v_1, v_{k+1}$ no vertices or edges are repeated, then $W$ is called a \emph{Berge cycle}.

Gerbner and Palmer \cite{GP2017} generalized the concepts of Berge paths and Berge cycles to general graphs in the study of Tur\'an type problem for hypergraphs.
Formally, let $F$ be a simple graph, and let $H$ be a hypergraph.
A hypergraph $H$ is called a \emph{Berge-$F$} \cite{GP2017} if $V(F) \subseteq V(H)$ and there is a bijection $\phi: E(F)\to E(H)$ such that $e\subseteq\phi(e)$ for each $e\in E(F)$.
Alternatively, $H$ is a Berge-$F$ if we can embed each edge of $F$ into a unique edge of $H$.
There is a common way to construct Berge hypergraphs by using expansion.
The \emph{$r$-expansion} of a simple graph $F$, denoted by $F^{r}$, is the $r$-uniform hypergraph obtained from $F$ by enlarging each edge of $F$ with a vertex set of size $r-2$ disjoint from $V(F)$ such that distinct edges are enlarged by disjoint vertex sets.
Note that for a fixed $F$ there are many hypergraphs that are Berge-$F$.
We use $\B_{r}(F)$ to denote the family of all $r$-uniform Berge-$F$ hypergraphs.
For example,  the family $\B_{3}(C_3)$ consists of $3$ hypergraphs in Fig. \ref{C3}, where $C_k$ denotes a cycle on $k$ vertices (as a simple graph).

\vspace{-2mm}
\begin{figure}[h]
\centering
\includegraphics[scale=.8]{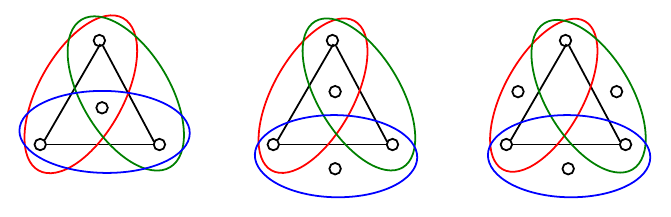}
\vspace{-2mm}
\caption{\small The family $\B_{3}(C_3)$, where the vertices in a colored closed curve form a hypergraph edge, and the vertices joined by a black line segment form a graph edge. }\label{C3}
\end{figure}

For a given family $\F$ of hypergraphs, we say $H$ is \emph{$\F$-free} if it does not contain a subhypergraph isomorphic to any hypergraph $F\in\F$.
Let $\ex_{r}(n,\F)$ and $\spex_{r}(n,\F)$ denote the maximum number of edges and the maximum spectral radius of $\F$-free $r$-uniform hypergraphs on $n$ vertices, respectively.
Similarly, write $\ex_{r}^{\text{lin}}(n,\F)$ and $\spex_{r}^{\text{lin}}(n,\F)$ for the maximum number of edges and the maximum spectral radius of $\F$-free linear
$r$-uniform hypergraphs on $n$ vertices, respectively.
If simple graphs are considered, the subscripts $r$ are omitted in the above notations.

Tur\'{a}n type problems on graphs and hypergraphs that aim to determine $\ex_{r}(n,\F)$ are the central topic of extremal combinatorics and have a vast literature; see, e.g. \cite{FS2013,MP2007,MS1965,Niki2002,Niki2006}.
The Erd\H{o}s-Stone-Simonovits theorem (\cite{Erdos1963,ES1966}) stated that $$\ex(n,F)=\left(1-\frac{1}{\chi(F)-1}\right)\binom{n}{2}+o(n^{2}),$$
 where $\chi(F)$ is the chromatic number of $F$.
However, it is a challenging problem to determine the order of magnitude of $\ex(n,F)$ for a bipartite graph $F$.

Let $K_{s,t}$ denote the complete bipartite graph with two parts having $s$ vertices and $t$ vertices respectively, where $s \le t$.
The famous Zarankiewicz problem just asks for the maximum number of edges in a $K_{s,t}$-free graph on $n$ vertices,
  and the matrix version of the Zarankiewicz problem is seeking for $z(m,n,s,t)$ that is the maximum number of ones in a $(0,1)$-matrix of size $m\times n$ that does not contain, $J_{s,t}$, the all-one matrix of size $s\times t$, as a submatrix.
Kóvari, Sós and Turán \cite{KST1955} presented one of the earliest bounds:
 \[\ex(n,K_{s,t})\le\frac{(t-1)^{\frac{1}{s}}}{2}n^{2-\frac{1}{s}}+\frac{s-1}{2}n.\]
 Füredi \cite{Furedi1996JCTA} gave the case of $s=2$:
 \[\lim_{n\to \infty}\ex(n,K_{2,t})n^{-\frac{3}{2}}=\frac{\sqrt{t-1}}{2}.\]
For the general case, Füredi \cite{Furedi1996CPC} improved the asymptotic coefficient $(t-1)^{\frac{1}{s}}$ to $(t-s+1)^{\frac{1}{s}}$:
 \[\ex(n,K_{s,t})\le\frac{(t-s+1)^{\frac{1}{s}}}{2}n^{2-\frac{1}{s}}+\frac{s}{2}n^{2-\frac{2}{s}}+\frac{s}{2}n.\]

 With the development of spectral extremal graph theory, the above results have spectral versions as well.
 Babai and Barry \cite{BG2009} proposed the spectral Zarankiewicz problem and gave the upper bound
 \[\spex(n,K_{s,t})\le\left((t-1)^{\frac{1}{s}}+o(1)\right)n^{1-\frac{1}{s}}, ~ 2\le s\le t.\]
 Furthermore, Nikiforov \cite{Niki2010} presented the following upper bounds:
 \[\spex(n,K_{2,t})\le \sqrt{(t-1)(n-1)+\frac{1}{4}} + \frac{1}{2}, \mbox{~for~} 2\le t,\]
  \[\spex(n,K_{s,t})\le(t-s+1)^{\frac{1}{s}}n^{1-\frac{1}{s}}+(s-1)n^{1-\frac{2}{s}}
  +s-2, \mbox{~for~} 3\le s\le t.\]

Due to the difficulty of a general hypergraph Tur\'an problem,
linear hypergraph Tur\'an type problems received more attention.
For example, Ruzsa and Szemer\'edi \cite{RS1978} considered the
$(6,3)$-problem, the maximum number of edges of $3$-uniform hypergraphs not carrying three edges on six vertices.
Erd\H{o}s, Frankl and R\"odl \cite{EFR1986} investigated the generalized $(6,3)$-problem, the maximum number of edges in an $r$-uniform hypergraph on $n$ vertices and in which the union of any three edges has size greater than $3r-3$ (equivalently, the number $\ex_{r}^{\text{lin}}(n,\B_{r}(C_{3}))$), and presented that \begin{equation}\label{6-3prob}
n^{2- \epsilon}<\ex_{r}^{\text{lin}}(n,\B_{r}(C_{3}))=o(n^{2}),
\end{equation}
for any $\epsilon> 0$ when $n$ is large enough.
Lazebnik and Verstra\"ete \cite{LV2003} gave the upper bound of the number of edges in an $r$-uniform hypergraphs $H$ on $n$ vertices without cycles of length less than five (implying that $H$ is linear and $\{\B_r(C_3),\B_r(C_4)\}$-free):
 \[e(H)\le\frac{1}{r(r-1)}n^{\frac{3}{2}}+\frac{r-2}{2r(r-1)}n+O(n^{-\frac{1}{2}}),\]
 where $e(H)$ denotes the number of edges of $H$;
 in particular, when $r=3$, if $H$ has the maximum number of edges, then
 \[e(H)=\frac{1}{6}n^{\frac{3}{2}}+o(n^{\frac{3}{2}}).\]
Timmons \cite{Timmons2016} generalized the result and presented
 \begin{equation}\label{Tim}
\ex_{r}^{\text{lin}}(n,\left\{\B_{r}(C_{3}),\B_{r}(K_{2,t})\right\})
  \le\frac{\sqrt{t-1}}{r(r-1)}n^{\frac{3}{2}}+\frac{n}{r},
  \end{equation}
  as $K_{2,2} \cong C_4$.
F\"uredi and \"Ozkahya~\cite{FO2017} proved that for $k\ge 2$,
   \[\ex_{3}^{\text{lin}}(n,\B_{r}(C_{2k+1}))\le 2kn^{1+\frac{1}{k}}+9kn.\]
Gerbner, Methuku and Vizer \cite{GMV2019} proved that for all $r,t\ge 2$, \[\ex_{r}^{\text{lin}}(n,\B_{r}(K_{2,t}))\le\frac{\sqrt{t-1}}{r(r-1)}n^{\frac{3}{2}}+O(n).\]
Ergemlidze, Győri and Methuku \cite{EGM2019} proved that \[\ex_{3}^{\text{lin}}(n,\B_{3}(C_{5}))=\frac{1}{3\sqrt{3}}n^{\frac{3}{2}}+O(n).\]
For all integers $r\ge 2$ and $2 \le s \le t$, Gao and Chang~\cite{GC2021} gave \[\ex_{r}^{\text{lin}}(n,K_{s,t}^{r})\le\frac{(t-1)^{\frac{1}{s}}}{r(r-1)}n^{2-\frac{1}{s}}+O(n^{2-\frac{2}{s}}),\]
     and
      \begin{equation}\label{Gao}\ex_{3}^{\text{lin}}(n,\left\{C_{3}^{3}, K_{s,t}^{3}\right\})\le
     \frac{(t-s+1)^{\frac{1}{s}}}{6}n^{2-\frac{1}{s}}
     +\frac{s-1}{6}n^{2-\frac{2}{s}}+\frac{s-2}{6}n.
     \end{equation}
Gao and Chang's result yields the corresponding bounds for some related Berge hypergraphs, as
 \[\ex^{\text{lin}}_{r}(n,\B_{r}(F))\le \ex^{\text{lin}}_{r}(n,F^{r}).\]

The spectral version of linear hypergraph Tur\'an type problems is seeking for the maximum spectral radius of $n$-vertex $\F$-free linear $r$-uniform hypergraphs for some given hypergraph family $\F$,
where the spectral radius refers to the maximum modulus of the eigenvalues of the adjacency tensor of a uniform hypergraph.
Hou, Chang and Cooper \cite{HCC2019} proved that
 \begin{equation}\label{Hou}
\spex_{r}^{\text{lin}}(n,\B_{r}(C_{4}))\le\sqrt{\frac{3}{2}+\frac{1}{2(r-1)}}(n-1)^{\frac{1}{2}}+O(n^{-\frac{1}{2}}).
\end{equation}
Gao, Chang and Hou~\cite{GCH2022} proved
\[\spex_{r}^{\text{lin}}(n,K_{r+1}^{r})\le \frac{n}{r},\]
with equality if and only if $r \mid n$ and a $K_{r+1}^{r}$-free linear $r$-uniform  hypergraph whose spectral radius attains the upper bound is a transversal design with $n$ vertices and $r$ groups.
She, Fan, Kang and Hou~\cite{SFKH2023} generalized the result to $F^{r}$ when $\chi(F)\ge r+1\ge 3$:
 \[\spex_{r}^{\text{lin}}(n,F^{r})=\frac{1}{r-1}\left (1-\frac{1}{\chi(F)-1}\right)n+o(n).\]

In this paper, we give a spectral analog result to the generalized $(6,3)$-problem or the asymptotic value of $\spex^{\text{lin}}_r(n, \B_r(C_3))$ (see Theorem \ref{spec-63-prob}),
and give an upper bound for the spectral radius of $\B_{r}(K_{2,t})$-free linear hypergraphs (see Theorem \ref{spec-K2t}), which extends the work of Hou, Chang and Cooper \cite{HCC2019} (Eq. (\ref{Hou})).
We also present an upper bound for the spectral radius of
$\{\B_{r}(K_{s,t}),\B_{r}(C_{3})\}$-free linear hypergraphs for a general $s,t$ (see Theorem \ref{mainresult}), and consequently we get an upper bound for the number of edges of $\{\B_{r}(K_{s,t}),\B_{r}(C_{3})\}$-free linear hypergraphs (see Corollary \ref{main-coro}), which extends the work of Timmons \cite{Timmons2016} (Eq. (\ref{Tim})) and that of Gao and Chang \cite{GC2021} (Eq. (\ref{Gao})).
The main technique used in this paper is to establish a relationship between the spectral radius of a uniform hypergraph and its local structure via $1$-walks and $2$-walks.

\section{Preliminaries}
For positive integers $r$ and $n$,
a complex \emph{tensor} (also called \emph{hypermatrix} \cite{CD2012}) $\A=(a_{i_{1}i_{2}\dots i_{r}})$ of order $r$ and dimension $n$ refers to a multidimensional array $a_{i_{1}i_{2}\dots i_{r}}\in\C$ for all $i_{1},i_{2},\dots,i_{r}\in[n]$.
In 2005, Lim~\cite{Lim2005} and Qi~\cite{Qi2005} independently introduced the eigenvalues of tensors as follows.
If there exists a number $\la \in \C$ and a nonzero vector $\x\in\C^{n}$ such that
		\begin{equation}\label{ev}
					\A\x^{r-1}=\la\x^{[r-1]},
		\end{equation}
 then $\la$ is called an \emph{eigenvalue} of $\A$, and $\x$ is called an \emph{eigenvector} of $\A$ corresponding to the eigenvalue $\la$,
 where $\x^{[r-1]}:=(x_i^{r-1}) \in \C^n$,
 and $\A\x^{r-1} \in \C^{n}$ is defined by
 \[(\A\x^{r-1})_{i}=\sum_{i_{2},\ldots,i_{r} \in [n]} a_{ii_{2}\dots i_{r}} x_{i_{2}}\cdots x_{i_{r}}, ~ i \in [n].\]
  The \emph{spectral radius} of $\A$, denoted by $\rho(\A)$, is defined to be the maximum modulus of the eigenvalues of $\A$.
	
In 2012, Cooper and Dutle~\cite{CD2012} introduced the \emph{adjacency tensor} of an $r$-uniform hypergraph $H$ with vertex set $V(H)=\{v_{1},v_{2},\dots,v_{n}\}$, which is an order $r$ dimension $n$ tensor $\A(H)=(a_{i_{1}i_{2}\dots i_{r}})$ whose $(i_{1},i_{2},\dots,i_{r})$-entry is given by
\[a_{i_{1}i_{2}\dots i_{r}}=
\begin{cases}\frac{1}{(r-1)!}, &\text{if }  \{v_{i_{1}},v_{i_{2}},\dots,v_{i_{r}}\}\in E(H);\\~ 0, &\text{otherwise}.
\end{cases} \]

The weak irreducibility of nonnegative tensors was defined by Friedland, Gaubert and Han~\cite{FGH2013}.
It was proved that an $r$-uniform hypergraph $H$ is connected if and only if its adjacency tensor $\A(H)$ is weakly irreducible (see~\cite{FGH2013} and~\cite{YY2011}).
The spectral property of nonnegative tensors or connected hypergraphs was investigated \cite{CD2012,FBH2019,FHB2022,FHBZL2019}
The Perron-Frobenius theorem of nonnegative matrices was generalized to nonnegative tensors, part of which is stated below.
	
\begin{theorem}[\cite{CPT2008,YY2011}]\label{PFthm}
Let $\A$ be a nonnegative tensor of order $r$ and dimension $n$. Then the following statements hold.
\begin{itemize}	
\item[{\rm (1)}] $\rho(\A)$ is an eigenvalue of $\A$ corresponding to a nonnegative eigenvector.
		
\item[{\rm (2)}] If furthermore $\A$ is weakly irreducible, then $\rho(\A)$ is the unique eigenvalue of $\A$ corresponding to the unique positive eigenvector up to a positive scalar.
\end{itemize}
\end{theorem}
	
\begin{definition}[\cite{CTS2023}]\label{2shadow}
For a hypergraph $H=(V,E)$, the \emph{$2$-shadow of $H$}, denoted by $\partial H$,  is the multigraph formed by replacing each edge $e$ of $H$ with a complete graph on the vertices of $e$.
That is,
\[\partial H=\left(V,\partial E,\varphi\right),\]
where $\partial E=\left\{\binom{e}{2}:e\in E(H)\right\} $ and $\varphi$ is a multiplicity function on $\partial E$ that satisfies
\[ \varphi(f)=\left|\left\{e\in E(H):f\subseteq e\right\}\right| \] for any $f\in\partial E$.
\end{definition}

For a multigraph $G=(V,E,\varphi)$, where $\varphi$ is a  multiplicity function on $E$,
the \emph{adjacency matrix} $A(G)=(a_{uv})$ is the matrix of order $\left|V\right|\times\left|V\right|$ with entries
 \[ a_{uv}=\begin{cases}
	\varphi(\{u,v\}), & \text{if }u\ne v \text{~and } \{u,v\} \in E; \\
	~ 0, & \text{otherwise}.
\end{cases} \]
In this paper, the \emph{spectral radius} and the \emph{eigenvectors} of a uniform hypergraph or a multigraph $H$ are referring to its adjacency tensor or adjacency matrix, and the spectral radius of $H$ is denoted by $\rho(H)$.

Let $H$ be a hypergraph and let $v$ be a vertex of $H$.
The \emph{degree} of $v$, denoted by $d(v)$, is defined to be the number of edges of $H$ that contain the vertex $v$.
The hypergraph $H$ is called \emph{regular} if all its vertices have the same degree.
Using a method similar to that in~\cite{SFKH2023}, we establish a relation between the spectral radius of a uniform hypergraph and that of its $2$-shadow graph.

\begin{lemma}[\cite{SFKH2023}]\label{spec-2shadow}
Let $H$ be an $r$-uniform hypergraph.
Then \[\rho(H)\le\frac{1}{r-1}\rho(\p H).\]
If $H$ is also connected, then the equality holds if and only if $H$ is regular.
\end{lemma}

Let $u,v$ be two vertices of a hypergraph $H$.
Denote by $w_{k}(u,v)$ the number of $k$-walks of $H$ starting at $u$ ending at $v$.
Therefore, the number of $k$-walks starting at $u$ is $w_{k}(u):=\sum_{v\in V(H)}w_{k}(u,v)$.
We apply the local structure of a uniform hypergraph to bound the spectral radius of the hypergraph.

\begin{lemma}\label{spec-ineq-walk}
 Let $H$ be an $r$-uniform hypergraph with spectral radius $\rho:=\rho(H)$.
 If $w_{2}(u)\le Pw_{1}(u)+(r-1)Q$ for any vertex $u\in V(H)$,
 then \[\rho^{2}-\frac{P}{r-1}\rho-\frac{Q}{r-1}\le 0,\]
 where $P$ and $Q$ are parameters independent of the choice of $u$.
\end{lemma}

\begin{proof} Let $\x$ be a nonnegative eigenvector of $H$ corresponding to the spectral radius $\rho$.
Denote $E_{v}:=\left\{e\in E(H):v\in e\right\}$ for a vertex $v$ of $H$, and
$\x^S:=\prod_{v \in S}x_v$ for a subset $S \subseteq V(H)$.
By the eigenvector equation (\ref{ev}), for any $v\in V(H)$,
$\rho x_{v}^{r-1}=\sum_{e\in E_{v}}\textbf{x}^{e\setminus\left\{v\right\}}$.
 Hence,
 \[\begin{aligned}
 \left(\rho-\frac{P}{r-1}\right)\sum_{v\in V(H)}\rho x_{v}^{r-1}
&=\left(\rho-\frac{P}{r-1}\right)\sum_{v\in V(H)}\sum_{e\in E_{v}}\textbf{x}^{e\setminus\left\{v\right\}}\\
& \le\left(\rho-\frac{P}{r-1}\right)\sum_{v\in V(H)}\sum_{e\in E_{v}}\sum_{u\in e\setminus\left\{v\right\}}\frac{x_{u}^{r-1}}{r-1}\\
&=\frac{1}{r-1}\sum_{v\in V(H)}\sum_{e\in E_{v}}\sum_{u\in e\setminus\left\{v\right\}}\left(\rho-\frac{P}{r-1}\right)x_{u}^{r-1}\\
&=\frac{1}{r-1}\sum_{v\in V(H)}\sum_{e\in E_{v}}\sum_{u\in e\setminus\left\{v\right\}}\left(\sum_{e^{\prime}\in E_{u}}\textbf{x}^{e^{\prime}\setminus\left\{u\right\}}-\frac{P}{r-1}x_{u}^{r-1}\right)\\
&\le \frac{1}{r-1}\left(\frac{1}{r-1}\sum_{v\in V(H)}\sum_{e\in E_{v}}\sum_{u\in e\setminus\left\{v\right\}}\sum_{e^{\prime}\in E_{u}}\sum_{w\in e^{\prime}\setminus\left\{u\right\}}x_{w}^{r-1}\right.\\
& \hspace{2.5cm} \left.-\sum_{v\in V(H)}\sum_{e\in E_{v}}\sum_{u\in e\setminus\left\{v\right\}}\frac{P}{r-1}x_{u}^{r-1}\right),
\end{aligned}\]
where the 2nd and the 5th inequalities follow from the AM-GM inequality.
Clearly, for any given $v\in V(H)$,
\[\sum_{e\in E_{v}}\sum_{u\in e\setminus\left\{v\right\}}\sum_{e^{\prime}\in E_{u}}\sum_{w\in e^{\prime}\setminus\left\{u\right\}}x_{w}^{r-1}=\sum_{w\in V(H)}w_{2}(v,w)x_{w}^{r-1},\]
\[\sum_{e\in E_{v}}\sum_{u\in e\setminus\left\{v\right\}}x_{u}^{r-1}=\sum_{u\in V(H)}w_{1}(v,u)x_{u}^{r-1}.\]
Thus, by the assumption of the theorem,
\[\begin{aligned}
\rho\left(\rho-\frac{P}{r-1}\right)\sum_{v\in V(H)}x_{v}^{r-1}
&\le\frac{1}{(r-1)^{2}}\left(\sum_{v\in V(H)}\sum_{w\in V(H)}w_{2}(v,w)x_{w}^{r-1}-\sum_{v\in V(H)}\sum_{u\in V(H)}Pw_{1}(v,u)x_{u}^{r-1}\right)\\
&=\frac{1}{(r-1)^{2}}\left(\sum_{w\in V(H)}\sum_{v\in V(H)}w_{2}(v,w)x_{w}^{r-1}-\sum_{u\in V(H)}\sum_{v\in V(H)}Pw_{1}(v,u)x_{u}^{r-1}\right)\\
&=\frac{1}{(r-1)^{2}}\left(\sum_{w\in V(H)}w_{2}(w)x_{w}^{r-1}-\sum_{u\in V(H)}Pw_{1}(u)x_{u}^{r-1}\right)\\
&\le\frac{Q}{r-1}\sum_{v\in V(H)}x_{v}^{r-1},
\end{aligned}\]
which implies the desired inequality.
\end{proof}

We remark that the parameters $P,Q$ in Lemma \ref{spec-ineq-walk} always exist.
For example, taking $P=\Delta(H)(r-1)$ and $Q$ any positive number, we have
\[\begin{aligned} w_2(u) &=\sum_{(e,v): uev \in W_1(u,v)} d(v)(r-1) \\
& \le \sum_{(e,v): uev \in W_1(u,v)} \Delta(H)(r-1) =\Delta(H)(r-1)w_1(u) \\
&\le P w_1(u)+(r-1)Q,
\end{aligned}
\]
where $\Delta(H)$ is the maximum degree of $H$, and $W_1(u,v)$ is the set of $1$-walks of $H$ from $u$ to $v$.
In fact, we wish to find $P,Q$ as small as possible.

Let $u,v$ be two vertices of a hypergraph $H$.
The \emph{distance} between $u$ and $v$ (or from $u$ to $v$) is the minimum length of the walks between them.
Denote respectively by $N_H(u)$ and $N_H^{2}(u)$ the sets of vertices of $H$ whose distance from $u$ is exactly one and two.
We will use $N(u)$ and $N^2(u)$ if there is no confusion.

\begin{corollary}\label{spec-ineq-degree}
	Let $H$ be a linear $r$-uniform hypergraph with the spectral radius $\rho:=\rho(H)$. If $\sum_{u\in N_{H}(v)}d(u)\le Pd(v)+Q$ for any vertex $v\in V(H)$,
then \[\rho^{2}-\frac{P}{r-1}\rho-\frac{Q}{r-1}\le 0,\]
where $P$ and $Q$ are parameters independent of the choice of $v$.
\end{corollary}

\begin{proof}
As $H$ is linear, for any $v\in V(H)$, $w_{1}(v)=(r-1)d(v)$, and
 \[w_{2}(v)=\sum_{u\in N_{H}(v)}(r-1)d(u)\le (r-1)Pd(v)+(r-1)Q=Pw_{1}(v)+(r-1)Q.\]
 The result follows by Lemma \ref{spec-ineq-walk}.
\end{proof}

We conclude this section with a lemma that will be used later.

\begin{lemma}{\cite{Furedi1996CPC}}\label{comb-ineq}
Let $n,k\ge 1$ be integers and $c,x_{0},x_{1},\dots,x_{n}$ be reals.
If $\sum_{i=1}^{n}\binom{x_{i}}{k}\le c\binom{x_{0}}{k}$,
then
\[\sum_{i=1}^{n}x_{i}\le x_{0}c^{\frac{1}{k}}n^{1-\frac{1}{k}}+(k-1)n.\]
\end{lemma}

\section{Spectral radius of $\B_r(C_3)$-free, $\B_r(K_{2,t})$-free, or $\{\B_{r}(C_{3}),\B_{r}(K_{s,t})\}$-free linear hypergraphs}\label{section 3}

Let $H$ be a hypergraph and let $S,T\subseteq V(H)$, where $S\cap T=\emptyset$.
Denote by $H[S]$ the subhypergraph of $H$ whose vertex set is $S$ and whose edge set consists of all edges of $H$ contained in $S$.
Similarly, $H[S,T]$ is the subhypergraph of $H[S\cup T]$ whose edge set consists of all edges of $H$ intersecting both $S$ and $T$.

\subsection{Spectral radius of $\B_r(C_3)$-free linear hypergraphs}

By Lemma \ref{spec-2shadow}, the spectral version of the generalized $(6,3)$-problem can be obtained.
Recall that by \cite[Theorem 3.8]{CD2012}, for an $r$-uniform hypergraph $H$ on $n$ vertices,
\begin{equation}\label{AveD}\rho(H) \ge \frac{r e(H)}{n},
\end{equation}
where the right side of the inequality is the average degree of $H$.

\begin{theorem}\label{spec-63-prob}
For any $\varepsilon>0$, when $n$ is large enough,
\begin{equation}\label{spec-63-ineq}
\begin{aligned}
n^{1-\varepsilon}<\spex_{r}^{\text{\rm lin}}(n,\B_{r}(C_{3}))=o(n).
\end{aligned}
\end{equation}
\end{theorem}	

\begin{proof}
By Eq. (\ref{6-3prob}), there exists a $\B_{r}(C_{3})$-free linear $r$-uniform hypergraph $H$ on $n$ vertices with more than $n^{2-\epsilon}$ edges.
So, by Eq. (\ref{AveD}),
\[
\spex^{\text{lin}}_{r}(n,\B_{r}(C_{3}))\ge
\rho(H) \ge \frac{r e(H)}{n}>rn^{1-\epsilon},
\]
which implies the lower bound in Eq. (\ref{spec-63-ineq}).

Next, we will prove the right equality.
Assume to the contrary that
$\spex^{\text{\rm lin}}_{r}(n,\B_{r}(C_{3}))\ne o(n)$.
Then there exist a real number $ \xi>0 $ and a sequence of $\B_{r}(C_{3})$-free linear $r$-uniform hypergraphs $H_{n_k}$ on $n_k$ vertices such that $\rho(H_{n_k})\ge \xi n_k$.
By Lemma \ref{spec-2shadow}, we have $\rho(\partial H_{n_k})\ge(r-1)\xi n_k$.
Without loss of generality, we assume that $\{H_n\}$ is a sequence of $\B_{r}(C_{3})$-free linear $r$-uniform hypergraphs on $n$ vertices such that $\rho(\partial H_{n})\ge\xi n$.
In the following, we simply write $H$ for $H_n$.

 Let $\x$ be a nonnegative eigenvector of $\p H$ corresponding to $\rho:=\rho(\partial H)$ with a certain vertex $u$ satisfying
 \[x_{u}=\max\left\{x_{v}:v\in V(\partial H)\right\}=1.\]
 Since $H$ is linear, $\p H$ is a simple graph.
 By eigenvector equation, we have
  \[\rho=\rho x_{u}=\sum_{v\in N_{\p H}(u)}x_{v} \le n-1.\]
  Therefore,
  \[\begin{aligned} \rho^{2}&=\rho^{2}x_{u}=\rho\sum_{v\in N_{\p H}(u)}x_{v}=\sum_{v\in N_{\p H}(u)}\sum_{w\in N_{\p H}(v)}x_{w}\\
  &=\sum_{v\in N_{\p H}(u)}\left(\sum_{w\in N_{\p H}(v)\cap\left\{N_{\p H}(u)\cup\left\{u\right\}\right\}}x_{w}
  +\sum_{w\in N_{\p H}(v)\setminus\left\{N_{\p H}(u)\cup\left\{u\right\}\right\}}x_{w}\right).
  \end{aligned}\]
Observe that
\[\begin{aligned}\sum_{v\in N_{\p H}(u)}\sum_{w\in N_{\p H}(v)\cap\left\{N_{\p H}(u)\cup\left\{u\right\}\right\}}x_{w}
=\sum_{\left \{v,w\right\}\in E(\p H[N_{\p H}(u)])}(x_{v}+x_{w})+\sum_{v\in N_{\p H}(u)}x_{u}. \end{aligned}\]
Moreover,
\[\begin{aligned}
&\frac{1}{2}\sum_{\left\{v,w\right\}\in E(\p H[N_{\p H}(u)])}(x_{v}+x_{w})+\sum_{v\in N_{\p H}(u)}x_{u}+\sum_{v\in N_{\p H}(u)}\sum_{w\in N_{\p H}(v)\setminus\left\{N_{\p H}(u)\cup\left\{u\right\}\right\}}x_{w}\\
&\le |E(\p H[N_{\p H}(u)])|+|N_{\p H}(u)|+|E(\p H[N_{\p H}(u),N_{\p H}^{2}(u)])|\\
&\le |E(\p H)|.
\end{aligned}\]
Since $H$ is $\B_{r}(C_{3})$-free and linear, $\p H[N_{\p H}(u)]$ is the graph consisting of $d_u(H)$ disjoint cliques of size $r-1$. So
\[\begin{aligned}
\frac{1}{2}\sum_{\left\{v,w\right\}\in E(\p H[N_{\p H}(u)])}(x_{v}+x_{w})
&=\frac{1}{2}\sum_{v\in N_{\p H}(u)}d_{\p H[N_{\p H}(u)]}(v) ~ x_{v}\\
&=\frac{r-2}{2}\sum_{v\in N_{\p H}(u)}x_{v}\\&=\frac{r-2}{2}\rho.
\end{aligned}\]
Therefore, by the above discussion and the fact $\rho < n$, we have
\[\begin{aligned} e(\p H) & \ge \rho^{2}-\frac{1}{2}\sum_{\left\{v,w\right\}\in E(\p H[N_{\p H}(u)])}(x_{v}+x_{w})\\
&= \rho^{2}-\frac{r-2}{2}\rho\\
& \ge\xi^{2}n^{2}-\frac{(r-2) n}{2}
\end{aligned}\] for $n$ sufficiently large.
So,
$e(H)=e(\p H)/\binom{r}{2} \ne o(n^2)$; a contradiction to Eq. (\ref{6-3prob}).
\end{proof}

\subsection{Spectral radius of $\B_r(K_{2,t})$-free linear hypergraphs}
Ergemlidze, Gy\H{o}ri and Methuku \cite{EGM2019} presented an upper bound of $\ex_{3}^{\text{lin}}(n,B_{r}(C_{4}))$ by the following relation on degrees.
Let $H$ be an $n$-vertex $\B_{3}(C_{4})$-free linear $3$-uniform hypergraph.
Then, for any $v\in V(H)$, \[\sum_{u\in N_{H}(v)}d(u)\le 6d(v)+\frac{n}{2}.\]
Note that $C_4 \cong K_{2,2}$.
While considering $\ex_{r}^{\text{lin}}(n,\B_{r}(K_{2,t}))$, Gerbner, Methuku and Vizer \cite{GMV2019} found the following relation.
	
\begin{lemma}\cite{GMV2019}\label{K2t-degree}
Let $H$ be a $\B_{r}(K_{2,t})$-free linear $r$-uniform hypergraph.
Then, for any $v\in V(H)$,
\[\sum_{u\in N_{H}(v)}d(u)\le(2r^{2}-4r+1)td(v)+\frac{(t-1)n}{r-1}.\]
\end{lemma}

Before proving the main result, we need the following basic fact.

\begin{fact}\label{root}
{\rm (1)} If $x^2-px-q \le 0$ with $p>0$, $q>0$ and $\sqrt{q} \ge \frac{p}{2}$, then
$ x \le q^\frac{1}{2}+p^\frac{1}{2}q^\frac{1}{4}.$

{\rm (2)} If $x^{2}-px-q\le 0$ with $p>0$ and $q\ge 0$, then $x\le p+\frac{q}{p}$.
\end{fact}

\begin{theorem}\label{spec-K2t}
For $t \ge 2$, if $n \ge \frac{(2r^2-4r+1)^2}{4(t-2)}$, then
\[ \spex_{r}^{\text{lin}}(n,\B_{r}(K_{2,t}))
\le \frac{(t-1)^{\frac{1}{2} }}{r-1}n^{\frac{1}{2}}
+\frac{(2r^{2}-4r+1)^{\frac{1}{2} }(t-1)^{\frac{1}{4} }t^{\frac{1}{2} }}{r-1}n^{\frac{1}{4} }.
\]
\end{theorem}

\begin{proof}
For an $n$-vertex $\B_{r}(K_{2,t})$-free linear $r$-uniform hypergraph $H$ with $ \rho:=\rho(H)$,
by Lemma \ref{K2t-degree},
taking $P=(2r^{2}-4r+1)t$ and $Q=\frac{(t-1)n}{r-1}$ in Corollary \ref{spec-ineq-degree},
we have
\[\rho^{2}-\frac{(2r^{2}-4r+1)t}{r-1}\rho-\frac{(t-1)n}{(r-1)^2}\le 0.\]
So, if $n \ge \frac{(2r^2-4r+1)^2}{4(t-2)}$, we have
\[\rho\le\frac{(t-1)^{\frac{1}{2} }}{r-1}n^{\frac{1}{2}}+\frac{(2r^{2}-4r+1)^{\frac{1}{2} }(t-1)^{\frac{1}{4} }t^{\frac{1}{2} }}{r-1}n^{\frac{1}{4} },\]
by using Fact \ref{root}(1).
\end{proof}

\subsection{Spectral radius of $\{\B_r(K_{s,t}),\B_r(C_3)\}$-free linear hypergraphs}
	
An $r$-uniform hypergraph $H$ is called \emph{hm-bipartite} if its vertex set has a bipartition (called \emph{hm-bipartition}) $V(H)=V_{1}\cup V_{2}$
such that each edge of $H$ intersects $V_{1}$ into exactly one vertex and $V_2$ the other $r-1$ vertices ~\cite{QSH2015,HQ2014}.
In the above bipartition, $V_1$ is called the \emph{head part} and $V_2$ is called the \emph{mass part}; and if further $\left|V_{1}\right|=m$, $\left|V_{2}\right|=n$, then the above $H$ is called \emph{$(m,n)$-hm-bipartite}.
An $(m,n)$-hm-bipartite linear $r$-uniform hypergraph $H$ is called \emph{exact $\B_{r}(K_{s,t})$-free} if $H$ contains no subhypergraphs isomorphic to any hypergraph in $\B_{r}(K_{s,t})$ with the part of $s$ vertices of $K_{s,t}$ in the head part of $H$ and the remaining part of $t$ vertices  in the mass part of $H$.

\begin{lemma}\label{edge-Kst-0}
Let $H$ be an $(m,n)$-hm-bipartite exact $\B_{r}(K_{s,t})$-free linear $r$-uniform hypergraph with an hm-bipartition $V(H)=V_{1}\cup V_{2}$ satisfying $|V_{1}|=m, |V_{2}|=n$, where $t\ge 2$, $ s\ge 2 $.
Suppose further that $H$ is $\B_{r}(C_{3})$-free.
 Then
 \begin{equation}\label{edge-Kst-0-ineq}
e(H)\le\frac{(t-1)^{\frac{1}{s}}}{r-1}mn^{1-\frac{1}{s}}+\frac{s-1}{r-1}n.
\end{equation}
\end{lemma}
	
\begin{proof}
For a fixed set $\left\{v_{1},v_{2},\dots,v_{s}\right\}\subseteq V_{1}$,
set
\[N_{V_{2}}(v_{1},v_{2},\dots,v_{s})=\left\{v\in V_{2}:\left\{v_{1},v_{2},\dots,v_{s} \right\}\subseteq N_{H}(v)\right\}.\]
Since $H$ is linear, if there is an edge containing $\left\{a,b\right\}$, then $\left\{a,b\right\}$ is contained in a unique edge of $H$, denoted by $e_{ab}$.
 Clearly, if $v_{0}\in N_{V_{2}}(v_{1},v_{2},\dots,v_{s})$, then $e_{v_{0}v_{1}},e_{v_{0}v_{2}},\dots,e_{v_{0}v_{s}}$ are mutually distinct edges as $H$ is hm-bipartite.
 If $v_{0},v_{0}^{\prime}\in N_{V_{2}}(v_{1},v_{2},\dots,v_{s})$, then $e_{v_{0}v_{1}},e_{v_{0}v_{2}},\dots,e_{v_{0}v_{s}},e_{v_{0}^{\prime }v_{1}}$, $e_{v_{0}^{\prime }v_{2}},\dots,e_{v_{0}^{\prime }v_{s}}$ are mutually distinct edges that form a Berge-$K_{s,2}$ as $H$ is hm-bipartite;
otherwise, it can be assumed that $e_{v_0 v_1}=e_{v'_0 v_1}=:e$, then $e,e_{v_{0}v_{2}},e_{v_{0}^{\prime }v_{2}}$ would form a Berge-$C_{3}$, which yields a contradiction to the assumption on $H$.

 Now, consider the number of the pair
 \[\left\{(v,\left\{v_{1},v_{2},\dots,v_{s}\right\}):v\in V_{2},\left\{v_{1},v_{2},\dots,v_{s}\right\}\subseteq N_{H}(v) \cap V_1 \right\}.
 \]
 By a simple double counting, noting that $H$ is hm-bipartite and exact $\B_r(K_{s,t})$-free,
 if $m\ge s$,
  \[\sum_{v\in V_{2}}\binom{d(v)}{s}\le (t-1)\binom{m}{s}.\]
  By Lemma \ref{comb-ineq}, we have
   \[(r-1)e(H)=\sum_{v\in V_{2}}d(v)\le(t-1)^{\frac{1}{s}}mn^{1-\frac{1}{s}}+(s-1)n, \]
which implies the upper bound in Eq. (\ref{edge-Kst-0-ineq}).

If $m\le s-1$, surely $H$ is exact $\B_r(K_{s,t})$-free.
As $H$ is an $(m,n)$-hm-bipartite and linear $r$-uniform hypergraph,
  each vertex $v$ of $V_1$ has degree $d(v) \le \frac{n}{r-1}$.
So,
\[
e(H) = \sum_{v \in V_1} d(v) \le \frac{mn}{r-1} \le \frac{s-1}{r-1}n,
\]
which also implies the upper bound in Eq. (\ref{edge-Kst-0-ineq}).
\end{proof}

\begin{theorem}\label{edge-Kst-k}
Let $H$ be an $(m,n)$-hm-bipartite exact $\B_{r}(K_{s,t})$-free linear $r$-uniform hypergraph with an hm-bipartition $V(H)=V_{1}\cup V_{2}$ satisfying $|V_{1}|=m, |V_{2}|=n$, where $s\ge 2$, $t\ge 2$.
Suppose further that $H$ is $\B_{r}(C_{3})$-free.
Then for  $k=0,1,\ldots,t-2$,
\begin{equation}\label{BKstC3}
e(H)\le\frac{(t-k-1)^{\frac{1}{s}}}{r-1}mn^{1-\frac{1}{s}}+\frac{s-1}{r-1}n^{1+\frac{k}{s}}+km.
\end{equation}
\end{theorem}
	
\begin{proof}
We will use induction on $k$. For $k=0$, the assertion is given by Lemma \ref{edge-Kst-0}. Suppose $k\ge 1$ and assume that the assertion is true for all $k^{\prime}<k$.
	
For any vertex $u\in V_2$, write $N_{V_1}(u)=N_H(u) \cap V_1$ and set
 \[H_{u}:=H[N_{V_{1}}(u)\cup V_{2}\setminus\left\{u\right\}]. \]
 Clearly, $d(u)=|N_{V_{1}}(u)|$ as $H$ is linear and hm-bipartite.
Further $H_{u}$ is an $(d(u), n-1)$-hm-bipartite exact $\B_{r}(K_{s,t-1})$-free linear $r$-uniform hypergraph; otherwise, $H$ would not be exact $\B_{r}(K_{s,t})$-free.
Surely $H_u$ is still $\B_r(C_3)$-free.

Note that $k\ge 1$ and therefore $t-1 \ge 2$.
By the induction assumption applied for $t-1$ and $k-1$ in $H_{u}$,
\begin{equation}\label{KstC3-1}
\begin{aligned}
e(H_{u})& \le\frac{(t-k-1)^{\frac{1}{s}}}{r-1}d(u)(n-1)^{1-\frac{1}{s}}
+\frac{s-1}{r-1}(n-1)^{1+\frac{k-1}{s}}+(k-1)d(u)\\
& \le \frac{(t-k-1)^{\frac{1}{s}}}{r-1}d(u)n^{1-\frac{1}{s}}
+\frac{s-1}{r-1}n^{1+\frac{k-1}{s}}+(k-1)d(u).
\end{aligned}
\end{equation}
On the other hand, as $H$ is hm-bipartite,
\begin{equation}\label{KstC3-2}
e(H_{u})=e\left(H[N_{V_{1}}(u)\cup V_{2}]\right)-d(u)=\sum_{v\in N_{V_{1}}(u)}d(v)-d(u).
\end{equation}
Thus, by Eqs. (\ref{KstC3-1}) and (\ref{KstC3-2}),
\begin{equation}\label{KstC3-3}
\sum_{v\in N_{V_{1}}(u)}d(v)\le\frac{(t-k-1)^{\frac{1}{s}}}{r-1}d(u)n^{1-\frac{1}{s}}
+\frac{s-1}{r-1}n^{1+\frac{k-1}{s}}+kd(u).
\end{equation}
Summing Eq. (\ref{KstC3-3}) for all $u\in V_{2}$, and noting that
$\sum_{u\in V_{2}}d(u)=(r-1)e(H)$,  we get
\[\begin{aligned}
  \sum_{u\in V_{2}}\sum_{v\in N_{V_{1}}(u)}d(v)
  &\le\frac{(t-k-1)^{\frac{1}{s}}}{r-1}\left(\sum_{u\in V_{2}}d(u)\right)n^{1-\frac{1}{s}}+\frac{s-1}{r-1}n^{2+\frac{k-1}{s}}+k\left(\sum_{u\in V_{2}}d(u)\right)\\
  &=(t-k-1)^{\frac{1}{s}}n^{1-\frac{1}{s}}e(H)+\frac{s-1}{r-1}n^{2+\frac{k-1}{s}}+k(r-1)e(H),
\end{aligned}\]
Meanwhile,
\[\begin{aligned}
   \sum_{u\in V_{2}}\sum_{v\in N_{V_{1}}(u)}d(v)
   &=\sum_{v\in V_1}\sum_{u \in N_{V_2}(v)}d(v)=(r-1)\sum_{v\in V_{1}}d(v)^{2}\\
   &\ge\frac{(r-1)\left(\sum_{v\in V_{1}}d(v)\right)^{2}}{m}=\frac{(r-1)e(H)^{2}}{m}.
\end{aligned}\]
Therefore,
\[e(H)^{2}-\left(\frac{(t-k-1)^{\frac{1}{s}}}{r-1}mn^{1-\frac{1}{s}}+km\right)e(H)-\frac{(s-1)mn^{2+\frac{k-1}{s}}}{(r-1)^{2}}\le 0. \]
So,
\[\begin{aligned}
e(H)&\le\frac{(t-k-1)^{\frac{1}{s}}}{r-1}mn^{1-\frac{1}{s}}+km
+\frac{\frac{(s-1)mn^{2+\frac{k-1}{s}}}{(r-1)^{2}}}{\frac{(t-k-1)^{\frac{1}{s}}}{r-1}mn^{1-\frac{1}{s}}+km}\\
&=\frac{(t-k-1)^{\frac{1}{s}}}{r-1}mn^{1-\frac{1}{s}}+km
+\frac{(s-1)n^{2+\frac{k-1}{s}}}{(r-1)\left((t-k-1)^{\frac{1}{s}}n^{1-\frac{1}{s}}+k(r-1)\right)}\\
&\le\frac{(t-k-1)^{\frac{1}{s}}}{r-1}mn^{1-\frac{1}{s}}+km+\frac{s-1}{r-1}n^{1+\frac{k}{s}},
\end{aligned}\]
where the first inequality applies Fact \ref{root}(2).
The result now follows.
\end{proof}

The right-hand side of the inequality (\ref{BKstC3}) is very interesting,
and the key is to see how much $k$ is taken to minimize the right-hand side of the inequality.
For more details, one can refer to \cite{Niki2010}.
If $2 \le s \le t$, letting $k=s-2$, we obtain the following corollary.
	
\begin{corollary}\label{edge-Kst-s-2}
Let $H$ be an $(m,n)$-hm-bipartite exact $B_{r}(K_{s,t})$-free linear $r$-uniform hypergraph with an hm-bipartition $V(H)=V_{1}\cup V_{2}$ satisfying $|V_{1}|=m, |V_{2}|=n$, where $2 \le s \le t$.
Suppose further that $H$ is $\B_{r}(C_{3})$-free.
Then
\[e(H)\le\frac{(t-s+1)^{\frac{1}{s}}}{r-1}mn^{1-\frac{1}{s}}+\frac{s-1}{r-1}n^{2-\frac{2}{s}}+(s-2)m.\]
\end{corollary}

Now we arrive at the main result in this subsection.
	
\begin{theorem}\label{mainresult}
Let $H$ be a $\left\{\B_{r}(K_{s,t}),\B_{r}(C_{3})\right\}$-free linear $r$-uniform hypergraph on $n$ vertices, where $2 \le s\le t$.
If $s=2$, then
\[\rho(H)\le\frac{\sqrt{4(t-1)(n-1)+(r-t)^{2}}+r-t}{2(r-1)}.\]
If $s\ge 3$, then
\[\rho(H)\le \frac{(t-s+1)^{\frac{1}{s}}}{r-1}n^{1-\frac{1}{s}}+\frac{s-1}{r-1}n^{1-\frac{2}{s}}+s-2.\]
\end{theorem}

\begin{proof}
First, consider the case of $s=2$.
 For any $u\in V(H)$, using double counting on the following set of pairs:
 $$ \{(v,w): v \in N(u), w \in N^2(u) \cap N(v)\},$$
and noting that $H$ is linear and $\left\{\B_{r}(K_{2,t}),\B_{r}(C_{3})\right\}$-free,
we have
\[\sum_{v\in N(u)}(r-1)(d(v)-1)\le(t-1)\left(n-1-(r-1)d(u)\right), \]
that is $$\sum_{v\in N(u)}d(v)\le(r-t)d(u)+\frac{(t-1)(n-1)}{r-1}.$$
By Corollary \ref{spec-ineq-degree}, we get $$\rho^{2}(H)-\frac{r-t}{r-1}\rho(H)-\frac{(t-1)(n-1)}{(r-1)^{2}}\le 0.$$
So,
 \[\rho(H)\le\frac{\sqrt{4(t-1)(n-1)+(r-t)^{2}}+r-t}{2(r-1)}. \]

 Next we consider the case of $s \ge 3$.
 For any $u\in V(H)$, set $H(u):=H[N_H(u),N_H^{2}(u)]$.
 Obviously $H(u)$ is an hm-bipartite $\B_{r}(C_{3})$-free linear $r$-uniform hypergraph with the head part $N_H(u)$ and the mass part $N_H^2(u)$.
 Furthermore, $H(u)$ is exact $\B_{r}(K_{s,t-1})$-free; otherwise $H$ contains a subhypergraph in $\B_{r}(K_{s,t})$.
 Note that
 $|N_H(u)|=(r-1)d(u)=:n_1, |N_H^2(u)|=:n_2 \le n-1-(r-1)d(u)$.
 So, $H(u)$ is an $(n_1,n_2)$-hm-bipartite exact $\B_{r}(K_{s,t-1})$-free $r$-uniform hypergraphs.
 As  $t\ge s\ge 3$, by Theorem \ref{edge-Kst-k},
 for every $0\le k\le t-3$,
  \[\begin{aligned}
 e\left(H(u)\right)& \le \frac{(t-k-2)^{\frac{1}{s}}}{r-1} n_1 n_2^{1-\frac{1}{s}}
 + \frac{s-1}{r-1}n_2^{1+\frac{k}{s}}+k n_1 \\
 & \le\left((t-k-2)^{\frac{1}{s}}n^{1-\frac{1}{s}}+k(r-1)\right)d(u)
 +\frac{s-1}{r-1}n^{1+\frac{k}{s}}.
 \end{aligned}\]
   Meanwhile, $e\left(H(u)\right)=\sum_{v\in N_{H}(u)}d(v)-(r-1)d(u)$.
  Thus,
  \[\sum_{v\in N_{H}(u) }d(v)\le\left((t-k-2)^{\frac{1}{s}}n^{1-\frac{1}{s}}+(k+1)(r-1)\right)d(u)
  +\frac{s-1}{r-1}n^{1+\frac{k}{s}}.
  \]
  By Corollary \ref{spec-ineq-degree}, \[\rho^{2}(H)-\frac{(t-k-2)^{\frac{1}{s}}n^{1-\frac{1}{s}}+(k+1)(r-1)}{r-1}\rho(H)-\frac{s-1}{(r-1)^{2}}n^{1+\frac{k}{s}}
  \le 0.
  \]
  Using Fact \ref{root}(2), we get \[\begin{aligned}
  \rho(H)&\le\frac{(t-k-2)^{\frac{1}{s}}n^{1-\frac{1}{s}}+(k+1)(r-1)}{r-1}+\frac{\frac{s-1}{(r-1)^{2}}n^{1+\frac{k}{s}}}{\frac{(t-k-2)^{\frac{1}{s}}n^{1-\frac{1}{s}}+(k+1)(r-1)}{r-1}}\\
  &\le\frac{(t-k-2)^{\frac{1}{s} }n^{1-\frac{1}{s}}}{r-1}+\frac{s-1}{r-1}n^{\frac{k+1}{s}}+k+1.
   \end{aligned}\]
  Setting $k=s-3$, we get the desired result.
\end{proof}

By Eq. (\ref{AveD}),  we obtain the following corollary immediately.

\begin{corollary}\label{main-coro}
For $t \ge 2$,
\begin{equation}\label{exk2t}
\ex_{r}^{\text{lin}}(n,\left\{\B_{r}(K_{2,t}),\B_{r}(C_{3})\right\}) \le \frac{\sqrt{4(t-1)(n-1)+(r-t)^{2}}+r-t}{2r(r-1)}n;
\end{equation}
and for $3\le s\le t$,
\begin{equation}\label{exkst}\ex_{r}^{\text{lin}}(n,\left\{\B_{r}(K_{s,t}),\B_{r}(C_{3})\right\})
\le
\frac{(t-s+1)^{\frac{1}{s}}}{r(r-1)}n^{2-\frac{1}{s}}+\frac{s-1}{r(r-1)}n^{2-\frac{2}{s}}+\frac{s-2}{r}n.
\end{equation}
\end{corollary}

If $(r-t)^2<4(t-1)$, then the upper bound in Eq. (\ref{exk2t})  will be smaller than Timmons' bound in Eq. (\ref{Tim}), but they have the same orders of magnitude.
If taking $r=3$ in the upper bound of Eq. (\ref{exkst}), we will get almost the same bound as in Eq. (\ref{Gao}) given by Gao and Chang only with a slight difference in the coefficient of $n$.
However, they also have the same orders of magnitude in this case.

When considering $\B_{r}(K_{s,t})$-free linear hypergraphs $H$ with removal of condition of $\B_r(C_3)$-free, the upper bound of $\rho(H)$ might become larger.
But can the asymptotic term be left unchanged?
We pose the following problem:

\begin{problem}
For $2 \le s\le t$, does the following inequality hold?
	$$\spex_{r}^{\text{lin}}(n,\B_{r}(K_{s,t}))\le\frac{(t-s+1)^{\frac{1}{s}}}{r-1}n^{1-\frac{1}{s}}+o(n^{1-\frac{1}{s}}).$$
\end{problem}

\section*{Declaration of Competing Interest}
The authors declared that they have no conflicts of interest to this work.

\section*{Data availability}
No data was used for the research described in the article.

\section*{Acknowledgement}
The authors are sincerely grateful to the referees for valuable suggestions to improve the paper.


\begin{thebibliography}{99}
\bibitem{BG2009}
L. Babai, B. Guiduli,
Spectral extrema for graphs: the Zarankiewicz problem,
\emph{Electron. J. Combin.}, 16 (2009) \#R123.

\bibitem{CPT2008}
K.-C. Chang, K. Pearson,  T. Zhang,
Perron-frobenius theorem for nonnegative tensors,
\emph{Commun. Math. Sci.}, 6 (2008) 507-520.

\bibitem{CTS2023}
G. J. Clark, F. Thomaz, A. T. Stephen,
Comparing the principal eigenvector of a hypergraph and its shadows,
\emph{Linear Algebra Appl.}, 673 (2023) 46-68.

\bibitem{CD2012}
J. Cooper, A. Dutle,
Spectra of uniform hypergraphs,
\emph{Linear Algebra Appl.}, 436 (2012) 3268-3292.

\bibitem{Erdos1963}
P. Erd\H{o}s,
On the structure of linear graphs,
\emph{Israel J. Math.}, 1 (1963) 156-160.

\bibitem{EFR1986}
P. Erd\H{o}s, P. Frankl, V. R\"odl,
The asymptotic number of graphs not containing a fixed subgraph and a
  problem for hypergraphs having no exponent,
\emph{Graphs Combin.}, 2 (1986) 113-121.

\bibitem{ES1966}
P. Erd\H{o}s, M. Simonovits,
A limit theorem in graph theory,
\emph{Studia Sci. Math. Hungar.}, 1 (1966) 51-57.

\bibitem{EGM2019}
B. Ergemlidze, E. Gy\H{o}ri, A. Methuku,
Asymptotics for Tur\'an numbers of cycles in 3-uniform linear
  hypergraphs,
\emph{J. Combin. Theory Ser. A}, 163 (2019) 163-181.


\bibitem{FBH2019} Y.-Z. Fan, Y.-H. Bao, T. Huang, Eigenvariety of nonnegative symmetric weakly irreducible tensors associated with spectral radius and its application to hypergraphs, \emph{Linear Algebra Appl.}, 564 (2019) 72-94.

\bibitem{FHB2022} Y.-Z. Fan, T. Huang, Y.-H. Bao,
The dimension of eigenvariety of nonnegative tensors associated with spectral radius, \emph{Proc. Amer. Math. Soc.}, 150 (2022) 2287-2299.


\bibitem{FHBZL2019}
Y.-Z. Fan, T. Huang, Y.-H. Bao, C.-L. Zhuan-Sun, Y.-P. Li,
The spectral symmetry of weakly irreducible nonnegative tensors and
  connected hypergraphs.
\emph{Trans. Amer. Math. Soc.}, 372 (2019) 2213-2233.



\bibitem{FGH2013}
S. Friedland, S. Gaubert, L. Han,
Perron-frobenius theorem for nonnegative multilinear forms and
  extensions,
\emph{Linear Algebra Appl.}, 438 (2013) 738-749.

\bibitem{Furedi1996CPC}
Z. F\"uredi,
An upper bound on Zarankiewicz' problem,
\emph{Combin. Probab. Comput.}, 5 (1996) 29-33.

\bibitem{Furedi1996JCTA}
Z. F\"uredi,
New asymptotics for bipartite Tur\'an numbers,
\emph{J. Combin. Theory Ser. A}, 75 (1996) 141-144.



\bibitem{FO2017}
Z. F\"uredi, L. \"Ozkahya,
On 3-uniform hypergraphs without a cycle of a given length,
\emph{Discrete Appl. Math.}, 216 (2017) 582-588.

\bibitem{FS2013}
Z. F\"uredi, M. Simonovits,
The history of degenerate (bipartite) extremal graph problems,
In \emph{\em Erd\H{o}s Centennial}, pages 169-264, Springer, 2013.

\bibitem{GC2021}
G. Gao, A. Chang,
A linear hypergraph extension of the bipartite Tur\'an problem,
\emph{European J. Combin.}, 93 (2021) 103269.

\bibitem{GCH2022}
G. Gao, A. Chang, Y. Hou,
Spectral radius on linear $r$-graphs without expanded $K_{r+1}$,
\emph{SIAM J. Discrete Math.}, 36 (2022) 1000-1011.

\bibitem{GMV2019}
D. Gerbner, A. Methuku, M. Vizer,
Asymptotics for the Tur\'an number of Berge-$K_{2,t}$,
\emph{J. Combin. Theory Ser. B}, 137 (2019) 264-290.

\bibitem{GP2017}
D. Gerbner, C. Palmer,
Extremal results for Berge hypergraphs,
\emph{SIAM J. Discrete Math.}, 31 (2017) 2314-2327.


\bibitem{HCC2019}
Y. Hou, A. Chang, J. Cooper,
Spectral extremal results for hypergraphs,
\emph{Electron. J. Combin.}, 28 (2021) \#P3.46.

\bibitem{HQ2014}
S. Hu, L. Qi,
The eigenvectors associated with the zero eigenvalues of the
  laplacian and signless laplacian tensors of a uniform hypergraph,
\emph{Discrete Appl. Math.}, 169 (2014) 140-151.

\bibitem{QSH2015}
L. Qi, J.-Y. Shao, S. Hu,
Some new trace formulas of tensors with applications in spectral
  hypergraph theory.
\emph{Linear Multilinear Algebra}, 63 (2015) 971-992.

\bibitem{KST1955}
T. K\'ovari, V. S\'os, P. Tur\'an,
 On a problem of K. Zarankiewicz,
\emph{Colloq. Math.}, 3 (1955) 50-57.

\bibitem{LV2003}
F. Lazebnik, J. Verstra\"ete,
 On hypergraphs of girth five,
\emph{Electron. J. Combin.}, 10 (2003) \#R25.

\bibitem{Lim2005}
L.-H. Lim,
 Singular values and eigenvalues of tensors: a variational approach,
 In \emph{1st IEEE International Workshop on Computational Advances in
  Multi-Sensor Adaptive Processing, 2005.}, pages 129-132, IEEE, 2005.

\bibitem{MS1965}
T. S. Motzkin, E. G. Straus,
 Maxima for graphs and a new proof of a theorem of Tur\'an,
\emph{Canad. J. Math.}, 17 (1965) 533–540.

\bibitem{MP2007}
D. Mubayi, O. Pikhurko,
 A new generalization of Mantel's theorem to $k$-graphs,
 \emph{J. Combin. Theory Ser. B}, 97 (2007) 669-678.

\bibitem{Niki2010}
V. Nikiforov,
 A contribution to the Zarankiewicz problem,
 \emph{Linear Algebra Appl.}, 432 (2010) 1405-1411.

\bibitem{Niki2002}
V. Nikiforov,
 Some inequalities for the largest eigenvalue of a graph,
 \emph{Combin. Probab. Comput.}, 11 (2002) 179–189.

\bibitem{Niki2006}
V. Nikiforov,
 Walks and the spectral radius of graphs,
\emph{Linear Algebra Appl.}, 418 (2006) 257-268.


\bibitem{Qi2005}
L. Qi,
 Eigenvalues of a real supersymmetric tensor,
 \emph{J. Symbolic Comput.}, 40 (2005) 1302-1324.

\bibitem{RS1978}
I. Z. Ruzsa, E. Szemer\'edi,
 Triple systems with no six points carrying three triangles,
 In \emph{\em Combinatorics (Proc. Fifth Hungarian Colloq.,
  Keszthely, 1976), Vol. II}, pages 939-945, Amsterdam-New York, 1978.

\bibitem{SFKH2023}
C.-M. She, Y.-Z. Fan, L.~Kang, Y. Hou,
 Linear spectral Tur\'an problems for expansions of graphs with
  given chromatic number,
   arXiv: 2211.13647v2, 2022.

\bibitem{Timmons2016}
C. Timmons,
 On $r$-uniform linear hypergraphs with no Berge-$K_{2,t}$,
 \emph{Electron. J. Combin.}, 24(2017) \#P4.34.

\bibitem{YY2011}
Y. Yang, Q. Yang,
 On some properties of nonnegative weakly irreducible tensors,
 arXiv: 1111.0713v2, 2011.

\end{thebibliography}
\end{document}